\documentclass{amsart}
\usepackage{amsmath}%
\usepackage{amsfonts}%
\usepackage{amssymb}%
\usepackage{graphicx}
\usepackage{mathtools}

\newcommand{\p}{\mathcal{P}}

\newcommand{\G}{\mathcal{G}}

\newcommand{\W}{\mathcal{W}}
\newcommand{\E}{\mathcal{E}}

\renewcommand{\S}{\mathcal{S}}

\newtheorem{theorem}{Theorem}[section]
\theoremstyle{plain}
\newtheorem{lem}[theorem]{Lemma}
\newtheorem{claim}[theorem]{Claim}

\newtheorem{rem}[theorem]{Remark}

\newtheorem{ques}{Question}
\newtheorem{example}{Example}

\numberwithin{equation}{subsection}
\theoremstyle{definition}
\newtheorem{defi}[theorem]{Definition}

\title{Lower bound on the number of non-simple geodesics on surfaces}
\author{Jenya Sapir}

\begin{document}

\begin{abstract}We give a lower bound on the number of non-simple closed curves on a
hyperbolic surface, given upper bounds on both length and self-intersection
number. In particular, we carefully show how to construct closed geodesics on
pairs of pants, and give a lower bound on the number of curves in this case.
The lower bound for arbitrary surfaces follows from the lower bound on pairs of
pants. This lower bound demonstrates that as the self-intersection number $K =
K(L)$ goes from a constant to a quadratic function of $L$, the number of closed
geodesics transitions from polynomial to exponential in $L$. We show upper
bounds on the number of such geodesics in a subsequent paper.
\end{abstract}
\maketitle

\section{Introduction}

\subsection{Statement of results}
Let $\p$ be a hyperbolic pair of pants with geodesic boundary, and let $\S$ be a hyperbolic genus $g$ surface with $n$ geodesic boundary components. Given a particular surface, let $\G^c$ be the set of closed geodesics. Set
\[
 \G^c(L,K) = \{ \gamma \in \G^c \ | \ l(\gamma) \leq L, i(\gamma, \gamma) \leq K \}
\]
where $l(\gamma)$ is geodesic length and $i(\gamma, \gamma)$ is geometric self-intersection number. We are interested in the following question:
\begin{ques}
 If $K = K(L)$ is a function of $L$, then how does $\#\G^c(L,K)$ grow with $L$?
\end{ques}

In this paper, we give a lower bound for growth. In future papers, we will get upper bounds on this number for an arbitrary surface, and tighter upper bounds on pairs of pants $\p$. The reason that the upper bounds are tighter on pairs of pants is that we have more control in how we construct geodesics there. In fact, we give a way to construct geodesics on pairs of pants in this paper.

Given a hyperbolic pair of pants, $\p$, let $l_{max} = l_{max}(\p)$ be the larger of the length of the longest boundary component of $\p$ or distance between boundary components of $\p$. We say that $l_{max}$ is the \textbf{length of $\p$}.

We get the following lower bound for a pair of pants $\p$:
\begin{theorem}
\label{thm:LowerBound}
 Let $\p$ be a hyperbolic pair of pants with length $l_{max}$, as defined above. If $L \geq 8 l_{max}$ and $K \geq 12$, we have that
\[
 \# \G^c(L,K) \geq 2 + \frac{1}{2} \min\{  2^{\frac{1}{8l_{max}}L},2^{\sqrt{\frac{ K}{ 12}}}\}
\]
\end{theorem}

A direct consequence of this theorem is the following lower bound for an arbitrary surface $\S$:
\begin{theorem}
\label{thm:LowerBoundSurface}
 Let $X$ be the hyperbolic metric on $\S$. Then whenever $K > 12$ and $L > 6 s_X \sqrt K$ we have 
 \[
 \# \G^c(L,K) \geq c_X \left (\frac{L}{6 \sqrt K} \right )^{6g-6+2n}2^{\sqrt{\frac{ K}{ 12}}}
 \]
 where $s_X$ and $c_X$ are constants that   depend only on the metric $X$.
\end{theorem}
The constant $s_X$ is roughly the width of the collar neighborhood of the systole of $X$, and $c_X$ is a constant related to the number of pairs of pants in $\S$ whose total boundary length is at most $L$.

Theorem \ref{thm:LowerBoundSurface} demonstrates that as $K = K(L)$ goes from a constant to a quadratic function in $L$, the number of closed geodesics on $\S$ transitions from polynomial to exponential in $L$.
(See Section \ref{sec:PreviousResults} for why we should expect such a transition.)

If $K$ is a constant, and $L$ is very large, this theorem says 
\[
\#\G^c(L,K) \geq c_X(K) L^{6g-6+2n}
\]
for $c_X(K)$ a new constant depending on $X$ and the constant $K$. This is consistent with the asymptotic results in \cite{Mirzakhani08,Rivin12} when $K =0$ and 1.  

By \cite{Basmajian13}, for any $\gamma \in \G^c$, $i(\gamma, \gamma) \leq \kappa l(\gamma)^2$, where $\kappa$ is a constant depending only on the metric. So we only need to consider functions $K(L)$ that grow at most like $O(L^2)$. For $K = O(L^2)$, however, we have that $\frac{L}{6 \sqrt K} = O(1)$, and Theorem \ref{thm:LowerBoundSurface} gives an exponential lower bound on $\#\G^c(L,K)$ in $L$. For example, if $K(L) = L^2$, then
\[
 \# \G^c(L,L^2) \geq c'_X 2^{\frac{L}{12}}
\]
where $c'_X$ is a new constant depending only on $X$.  This is consistent with the growth of all closed geodesics with length at most $L$ in \cite{MargulisPhD}. 

\subsection{Previous results}
\label{sec:PreviousResults}
The problem of counting closed geodesics in many contexts has been studied extensively. There is an excellent survey of the history of this problem by Richard Sharp that was the  published in conjunction with Margulis's thesis in \cite{Margulis04}. The following is a brief, but incomplete, overview. 

Let
\[
 \G^c(L) = \{\gamma \in \G^c \ | \ l(\gamma) \leq L\}
\]
The famous result in Margulis's thesis states that if $\S$ is negatively curved with a complete, finite volume metric, then
\begin{equation}
\label{GL}
 \#\G^c(L) \sim \frac{e^{\delta L}}{\delta L}
\end{equation}
where $\delta$ is the topological entropy of the geodesic flow, and where $f(L) \sim g(L)$ if $\lim_{L \rightarrow \infty}\frac{f(L)}{g(L)} = 1$ \cite{MargulisPhD}. (Note that $\delta = 1$ when $\S$ is hyperbolic.) A version of this result for hyperbolic surfaces was first proven by Huber \cite{Huber59}. There are also many other, later versions of this result for non-closed surfaces. For example, see \cite{Colin85,Patterson88,LP82,Lalley89} and \cite{Guillope86}.

Recently, there has been work on the dependence of the number of closed geodesics on their self-intersection number as well as length.
The goal is to answer the following question.
\begin{ques}
 If $K = K(L)$ is a function of $L$, what is the asymptotic growth of $\#\G^c(L,K)$ in terms of $L$?
\end{ques}

As part of her thesis, Mirzakhani showed that for a hyperbolic surface $\S$ of genus $g$ with $n$ punctures,
\[
 \#\G^c(L,0) \sim c(\S) L^{6g-6 + 2n}
\]
where $c(\S)$ is a constant depending only on the geometry of $\S$ \cite{Mirzakhani08}. Rivin extended this result to geodesics with at most one self-intersection, to get that 
\[
  \# \G^c(L,1) \sim c'(\S) L^{6g-6 +2n}
\]
where $c'(\S)$ is another constant depending only on the geometry of $\S$ \cite{Rivin12}. It should be noted that polynomial upper and lower bounds on $\#\G^c(L,0)$ were first shown by Rees in \cite{Rees81}.

However, no asymptotics are yet known for arbitrary functions $K(L)$. This paper, and the one that follows, give bounds on $\#\G^c(L,K)$ for $L$ and $K$ large enough.

\subsection{Idea of proof}
Theorem \ref{thm:LowerBoundSurface} is a direct consequence of Theorem \ref{thm:LowerBound}, which is proven as follows:
\begin{itemize}
 \item In Section \ref{sec:CombModel}, we create a combinatorial model for geodesics on a pair of pants. Each geodesic $\gamma$ can be represented as a cyclic word $w(\gamma)$ in a finite alphabet (Lemma \ref{lem:ProjectionConstruction}). We then give some basic properties of these words in Section \ref{sec:WordStructure}.
 
 These words are also the key ingredient in getting an upper bound on $\# \G^c(L,K)$ for pairs of pants. This is done in a subsequent paper.
 \item We show that if $w = w(\gamma)$, then
 \[
  l(\gamma) \asymp |w|
 \]
 where $|w|$ denotes word length (Lemma \ref{lem:GeodesicAndWordLength}), and
 \[
  i(\gamma, \gamma) \leq i(w,w)
 \]
 (Lemma \ref{lem:IntersectionLowerBound}), where $i(w,w)$ is an intersection number for words defined in Definition \ref{def:IntForWords}.
 \item Then in Section \ref{sec:ConstructingGeodesics}, we construct a set of distinct geodesics. We show that each geodesic $\gamma$ arising from this construction lies in $\G^c(L,K)$ by bounding $|w(\gamma)|$ and $i(w(\gamma), w(\gamma))$ from above. 
 
 We get a lower bound on the number of geodesics we construct, giving us a lower bound on $\#\G^c(L,K)$. For a more detailed summary, see Section \ref{sec:ProofSummary}.
\end{itemize}

We prove Theorem \ref{thm:LowerBoundSurface} in Section \ref{sec:SurfaceBound}. We look at all geodesically embedded pairs of pants in $\S$ that have closed geodesics of length at most $L$. Summing $\#\G^c(L,K)$ over all these pairs of pants gives the theorem. 
%

This paper is part of the author's PhD thesis, which was completed under her advisor, Maryam Mirzakhani. The author would especially like to thank her for the many conversations that led to this work. The author would also like to thank Jayadev Athreya, Steve Kerckhoff and Chris Leininger for their help and support. 

\section[A combinatorial model]{A combinatorial model for geodesics on pairs of pants}
\label{sec:CombModel}
Let $\p$ be a hyperbolic pair of pants with geodesic boundary. In this section, we construct a combinatorial model for closed geodesics on $\p$, and show that this model allows us to recover geometric properties of the corresponding geodesics. We do this as follows. First, there is a unique way to write $\p$ as the union of two congruent right-angled hexagons. Take this decomposition (Figure \ref{fig:HexagonDecomp}).
\begin{figure}[h!] \centering
 \includegraphics{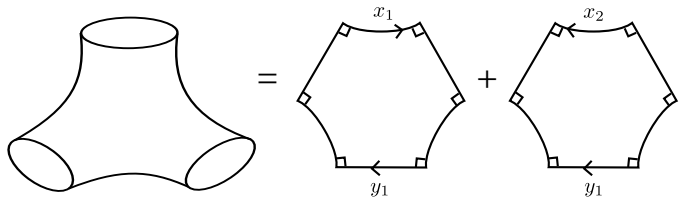}
\caption[Hexagon decomposition]{The hexagon decomposition of $\p$ with boundary edges $x_1$ and $x_2$ and seam edge $y_1$ labeled.}
\label{fig:HexagonDecomp}
\end{figure} 

Let $\E$ be the set consisting of two copies of each edge in the hexagon decomposition, one copy for each orientation. The set $\E$ consists of oriented edges $x_1, \dots, x_{12}$ that lie on the boundary of $\p$ and oriented edges $y_1, \dots, y_6$ that pass through the interior of $\p$. We call $x_1, \dots, x_{12}$ \textbf{boundary edges} and $y_1, \dots, y_6$ \textbf{seam edges}.

We can model closed geodesics on $\p$ by looking at closed concatenations of edges in $\E$. If $p$ is a closed concatenation of edges in $\E$, then it corresponds to a cyclic word $w$ with letters in $\E$. We want to look at the following subset of such words.

\begin{defi}
\label{defi:W}
 Let $\W$ be the set of cyclic words $w$ with letters in $\E$ so that
 \begin{itemize}
  \item The letters of $w$ can be concatenated (in the order in which they appear) into a closed path $p$.
  \item The curve $p$ does not back-track.
  \item Lastly, we want a technical condition: each $w \in \W$ can be written as $w = b_1 s_1 \dots b_n s_n$ where $b_i$ is a sequence of boundary edges, $|b_i| \geq 2$, and $s_i$ is a seam edge ($|s_i| = 1$) for each $i$, unless $n = 1$, in which case $w = b_1$.
  \end{itemize}
\end{defi}

Clearly, there is a map $\W \rightarrow \G^c$. For each $w \in \W$, we simply take the corresponding closed curve $p(w)$. Each closed curve on $\p$ has exactly one geodesic in its free homotopy class. Let $\gamma(w)$ be the geodesic in the free homotopy class of $p(w)$. Then the map $w \mapsto \gamma(w)$ is well-defined.

\subsection{Turning Geodesics into Words}
\label{sec:ProjectionOfCurve}

Conversely, we can explicitly construct a map going back. In fact, we can construct an \textit{injective} map $\G^c \rightarrow \W$ sending each geodesic $\gamma$ to some preferred word in $\W$. 




\subsubsection{The Projection $p(\gamma)$ of a Closed Geodesic $\gamma$} 
\label{sec:JustProjection}
For each closed geodesic $\gamma$, we first construct a closed curve $p(\gamma)$ that lies on the boundaries of the hexagons and is freely homotopic to $\gamma$. Then $p(\gamma)$ will be known as the \textbf{projection} of $\gamma$ to the edges in $\E$. We give the desired properties of $p(\gamma)$ in the following lemma. These properties will allow us to convert $p(\gamma)$ into a word $w(\gamma) \in \W$.

\begin{lem}[Construction of $p(\gamma)$] 
\label{lem:ProjectionConstruction}
Let $\gamma$ be a closed geodesic on $\p$. Then there is a closed curve $p(\gamma)$ that has the following properties:
\begin{enumerate}
\item $p(\gamma)$ is freely homotopic to $\gamma$.
\item $p(\gamma)$ is a concatenation of edges in $\E$.
\item each boundary edge in $p(\gamma)$ is concatenated to at least one other boundary edge.
\end{enumerate}
\end{lem}

\begin{proof}
 
Let $\gamma$ be an oriented, closed geodesic. The idea of the construction is given in the following four steps. See Figure \ref{fig:CurveProjection} for the accompanying illustration.
\begin{figure}[h!] \centering
 \includegraphics{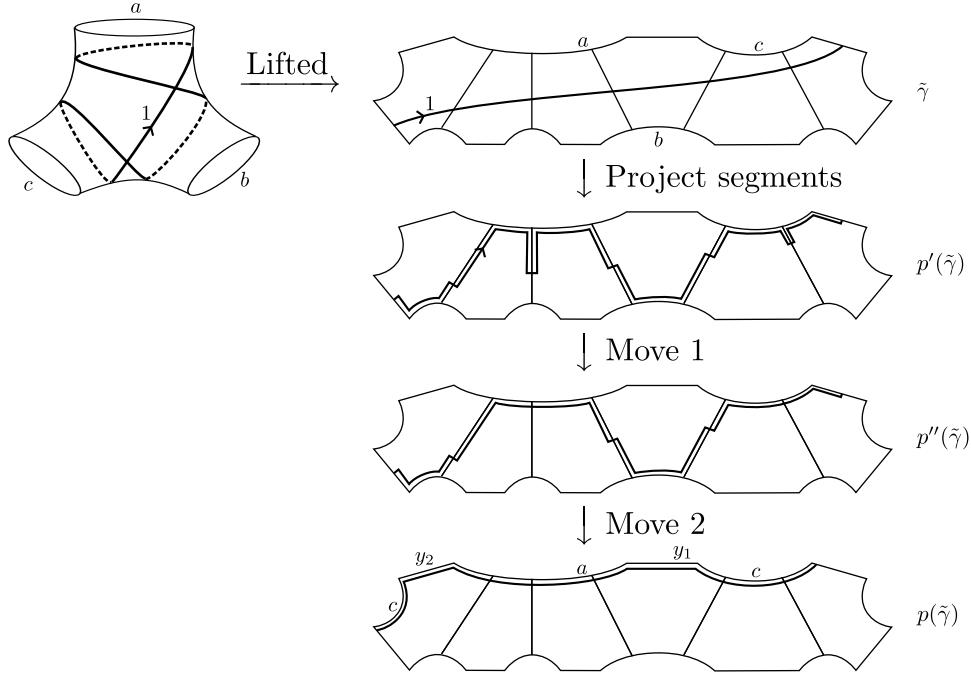}
\caption[Projection of a closed curve]{Projection of a closed curve. From top to bottom, we construct $\gamma$, $p'(\gamma)$, $p''(\gamma)$ and then finally $p(\gamma)$.}
\label{fig:CurveProjection}
\end{figure} 

\begin{enumerate}
 \item We break $\gamma$ up into segments, which are pieces of $\gamma$ that live entirely inside hexagons. 
 \item Each segment $\sigma$ lying inside a hexagon $h$ gets projected to a sub-arc $p'(\sigma)$ of the boundary of $h$ (Figure \ref{fig:SegmentHomotopy}). We can concatenate the arcs $p'(\sigma)$ to get a closed curve $p'(\gamma)$, which lies entirely in the boundaries of the two hexagons and is homotopic to $\gamma$. At this stage, $p'(\gamma)$ need not be the concatenation of edges in $\E$.
 \item We define a homotopy called Move 1 that we apply to finitely many disjoint sections of $p'(\gamma)$. The result is a curve $p''(\gamma)$ that is the concatenation of edges in $\E$.
 \item Lastly, we force each boundary edge to be concatenated to another boundary edge via a homotopy called Move 2, which we apply to sections of $p''(\gamma)$. This gives us a closed curve $p(\gamma)$ satisfying Lemma \ref{lem:ProjectionConstruction}.
\end{enumerate}

Now we fill in the details. Let a \textbf{segment} $\sigma$ of $\gamma$ be a maximal sub-arc that lies entirely in some hexagon $h$ of the hexagon decomposition of $\p$. The projection $p'(\sigma)$ of $\sigma$ is the shortest sub-arc of the boundary of $h$ that has the same endpoints as $\sigma$ and contains exactly one boundary edge. If the boundary edge is $x$, we will say that $\sigma$ is projected onto $x$ (Figure \ref{fig:SegmentHomotopy}.)

Since $\sigma$ and $p'(\sigma)$ have the same endpoints, and since $\gamma$ is the concatenation of segments, we can concatenate all of the arcs $p'(\sigma)$ into a closed curve $p'(\gamma)$. By construction, $p'(\gamma)$ is homotopic to $\gamma$.

\begin{figure}[h!] \centering
	\centering
		\includegraphics{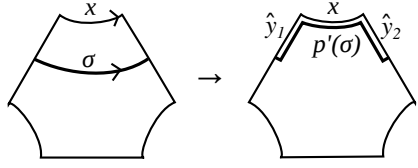}
	\caption[Projection of a segment]{Projecting a segment $\sigma$ onto the boundary edge $x$.}
	\label{fig:SegmentHomotopy}
\end{figure}

\textbf{Move 1}: The goal is to homotope $p'(\gamma)$ into a curve $p''(\gamma)$ that is the concatenation of edges in $\E$. This is needed in the situation in Figure \ref{fig:LastPush}.

\begin{figure}[h!] \centering
	\centering
		\includegraphics{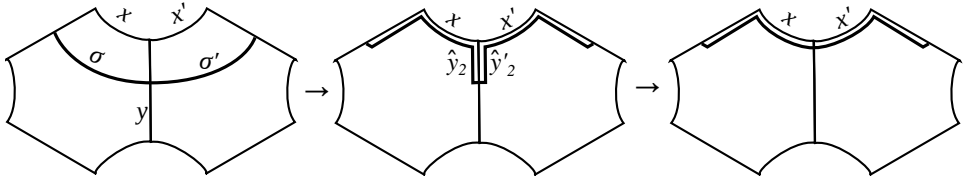}
	\caption[Move 1]{Move 1 is needed when $\sigma$ and $\sigma'$ project onto the same boundary component of $\p$.}
	\label{fig:LastPush}
\end{figure} 

\begin{figure}[h!] \centering
	\centering
		\includegraphics{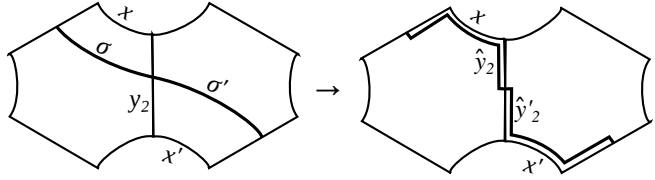}
	\caption[Move 1 not needed]{When $\sigma$ and $\sigma'$ project onto different boundary components of $\p$, Move 1 is not needed.}
	\label{fig:DifferentParallelSides}
\end{figure}

The orientation of $\gamma$ gives a cyclic ordering to its segments. Suppose segment $\sigma$ is followed by segment $\sigma'$. Write their projections as $p'(\sigma) = \hat y_1 \circ x \circ \hat y_2$ and $p'(\sigma') = \hat y_2' \circ x' \circ \hat y_3'$, where $x$ and $x'$ are boundary edges and $\hat y_i, \hat y_i'$ are pieces of the seam edge $y_i$ for each $i = 1,2,3$.  Note that because $\sigma$ and $\sigma'$ are consecutive segments, $\hat y_2$ and $\hat y_2'$ are both pieces of the same seam edge $y_2$. Furthermore, the endpoint of $\hat y_2$ is the start point of $\hat y_2'$. Thus the the concatenation $\hat y_2 \circ \hat y_2'$ is either null-homotopic relative to its endpoints (Figure \ref{fig:LastPush}) or it is all of $y_2$ (Figure \ref{fig:DifferentParallelSides}.)

Move 1 is to homotope away concatenations of the form $\hat y \circ \hat y'$ when they are null-homotopic. We apply it finitely many times to $p'(\gamma)$ to get a new closed curve $p''(\gamma)$ that is a concatenation of edges in $\E$. In fact, the number of times we must apply Move 1 is at most the number of segments in $\gamma$. Note that $p''(\gamma)$ is still homotopic to $\gamma$.

\textbf{Move 2}: 
\begin{figure}[h!] \centering
	\centering
		\includegraphics{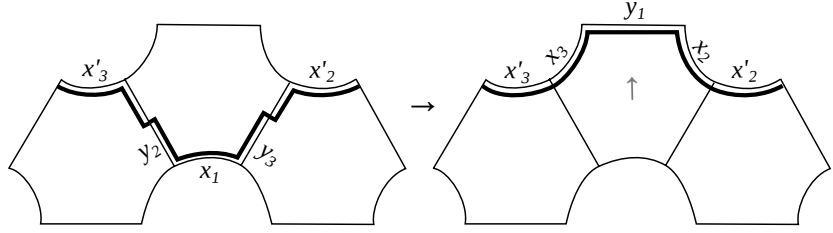}
	\caption[Move 2]{Move 2. The boundary edge $x_1$ on the left is isolated, while the boundary edges $x_2$ and $x_3$ on the right are not.}
	\label{fig:LastHHomotopy}
\end{figure}
Now we homotope $p''(\gamma)$ to a new curve $p(\gamma)$ in which each boundary edge is concatenated to another boundary edge (Figure \ref{fig:LastHHomotopy}.) If a boundary edge $x$ is not concatenated to any other boundary edge, we call $x$ an \textbf{isolated boundary edge}. 

Note that $p'(\gamma)$ never has more than 3 consecutive edges lying on the boundary of the same hexagon. Let $p''_1$ be any closed concatenation of edges in $\E$ with this property. We claim that we can homotope $p''_1$ to a curve $p''_2$ with strictly fewer isolated boundary edges.

If $p''_1$ has an isolated boundary edge $x_1$, then it has a subarc of the form $y_2 \circ x_1 \circ y_3$ lying on the boundary of a single hexagon $h$, where $y_2$ and $y_3$ are seam edges. We will homotope it relative its endpoints to the other part of the boundary of $h$. This is an arc of the form $x_3 \circ y_1 \circ x_2$, where $x_3$ and $x_2$ are boundary edges and $y_1$ is a seam edge:
\[
 y_2 \circ x_1 \circ y_3 \mapsto x_3 \circ y_1 \circ x_2
\]
This is \textbf{Move 2} (Figure \ref{fig:LastHHomotopy}). It gives us a new arc $p''_2$.

We claim that $p''_2$ has at least one fewer isolated boundary edge than $p''_1$. This is the same as showing that $x_2$ and $x_3$ are not isolated in $p''_2$. We have that $p''_1$ never follows more than three consecutive sides of a hexagon at a time. So $y_2$ and $y_3$ must be concatenated in $p''_1$ to edges in the other hexagon. These can only be the boundary edges $x_2'$ and $x_3'$ which lie on the same boundary components as $x_2$ and $x_3$, respectively. Thus $x_2$ and $x_3$ are not isolated.

Therefore, $p''_2$ has strictly fewer isolated boundary edges than $p''_1$. Since $p''(\gamma)$ has finitely many (isolated) boundary edges, we can perform Move 2 finitely many times to get a closed curve $p(\gamma)$ with no isolated boundary edges.

\begin{rem}
\label{rem:Move2Dependence}
Applying Move 2 can get rid of more than one isolated boundary edge at a time. Thus the final arc $p(\gamma)$ depends on the order in which we get rid of isolated boundary edges. For each closed geodesic $\gamma$, we make a choice of $p(\gamma)$ once and for all.
\end{rem}

\end{proof}


\subsubsection{Defining the Cyclic Word $w(\gamma)$ for a Closed Geodesic $\gamma$}
\label{sec:WofGamma}
By Remark \ref{rem:Move2Dependence}, we force the map $\gamma \mapsto p(\gamma)$ to be well-defined. Since $p(\gamma)$ is a concatenation of edges in $\E$, it corresponds to a cyclic word $w(\gamma)$ with letters in $\E$. Thus each $\gamma \in \G^c$ corresponds to a unique word $w(\gamma)$.

We show in Lemma \ref{lem:MostlyTwistWords} that $w(\gamma) \in \W$, where $\W$ is the set defined in Definition \ref{defi:W}.

\subsubsection{Injective Correspondence for Closed Geodesics}
\label{sec:InjectiveCorrespondence}

The most important relationship between closed geodesics in $\G^c$ and words in $\W$ is that distinct geodesics correspond to different words. This is because $\gamma$ is homotopic to $p(\gamma)$, so if two geodesics correspond to the same word, they are homotopic as well. Since there is exactly one geodesic in each free homotopy class of non-trivial closed curves, the map from closed geodesics to words is injective. We formalize this in the following remark.
\begin{rem}
 \label{rem:InjectiveCorrespondence}
 If $\gamma \neq \gamma'$ are two distinct closed geodesics on $\p$ then $w(\gamma) \neq w(\gamma')$.
\end{rem}



\subsection{Word Structure: Boundary Subwords}
\label{sec:WordStructure}

We want to examine the form of a word $w(\gamma)$ in more detail. A closed geodesic $\gamma \in \G^c$ spends most of its time twisting about boundary components of $\p$. So its projection $p(\gamma)$ spends most of its time winding around those boundary components. Note that to transition from one boundary component to another, $p(\gamma)$ only needs to take a single seam edge (Figure \ref{fig:ProjectionStructure}.) Thus, $w(\gamma)$ has long sequences of boundary edges (called boundary subwords) separated by single seam edges (Lemma \ref{lem:MostlyTwistWords}.) Furthermore, $w(\gamma)$ is completely determined by the sequence of boundary subwords that appear (Lemma  \ref{lem:WordDetermined}.) This is because there is at most one seam edge connecting one boundary edge to another.
\begin{figure}[h!] \centering
 \includegraphics{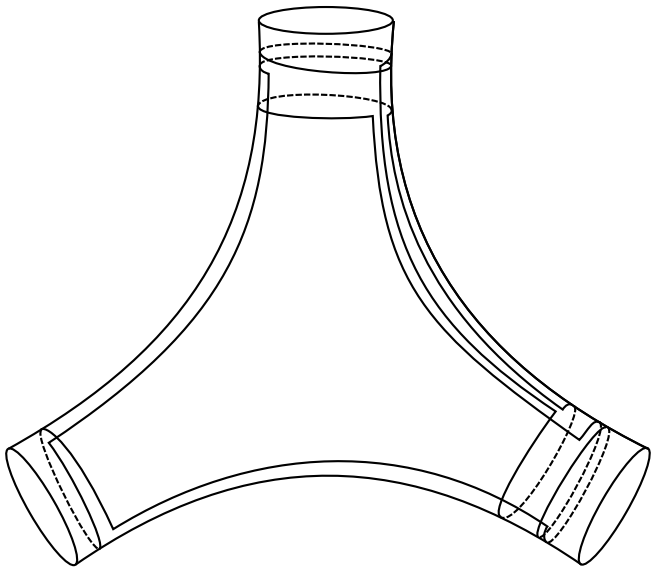}
 \caption{An approximation to $p(\gamma)$}
 \label{fig:ProjectionStructure}
\end{figure}

\begin{lem}
\label{lem:MostlyTwistWords}
For each $\gamma \in \G^c$, the word $w(\gamma)$ is in $\W$, where $\W$ is the set defined in Definition \ref{defi:W}.
\end{lem}

\begin{proof}
Let $\gamma \in \G^c$. By construction, $w(\gamma)$ can be concatenated into a closed curve $p(\gamma)$, and this curve $p(\gamma)$ does not back-track. We just need to show that 
\[
 w(\gamma) = b_1s_1 \dots b_ns_n
\]
where $b_i$ consists only of boundary edges, with $|b_i| \geq 2$ and $s_i$ is a single seam edge for each $i$, unless $n = 1$, in which case $w(\gamma) = b_1$.

Note that $w(\gamma)$ can alway be written to start with a boundary edge and end with a seam edge, unless $w(\gamma)$ consists only of boundary edges. (By construction, $w(\gamma)$ always contains at least 1 boundary edge.) So we can always write
\[
 w(\gamma) = b_1s_1 \dots b_n s_n
\]
where $b_i$ is a non-empty sequence of boundary edges, and where $s_i$ is a non-empty sequence of seam edges for each $i$, unless $w(\gamma) = b_1$. 

The condition that each boundary edge is concatenated to another boundary edge guarantees that $|b_i| \geq 2$ for each $i$. The fact that $p(\gamma)$ does not back-track guarantees that each seam edge can only be concatenated to a boundary edges. So $|s_i| = 1$ for each $i$. 


If $n = 1$, we want to show that $w(\gamma) = b_1$. But a (non-cyclic) word of the form $b_1s_1$ has endpoints on different boundary components of $\p$, and so does not close up into a closed curve. So if $n = 1$, then $w(\gamma) = b_1$.


\end{proof}

We now get a few more properties of the structure of $w(\gamma)$.

\begin{lem}
\label{lem:WordDetermined}
 A cyclic word $w(\gamma)$ is completely determined by its sequence $b_1, \dots, b_n$ of boundary subwords. That is, if $w(\gamma) = b_1 s_1 \dots b_n s_n$ and $w(\gamma') = b_1 s'_1 \dots b_n s'_n$ (where the boundary subwords of $w$ and $w'$ are the same and in the same order), then $s_i = s_i'$ for each $i$.
\end{lem}

\begin{proof}
This lemma follows from the fact that any two boundary subwords can be joined together by at most one seam edge. In particular, given two oriented boundary edges $x$ and $x'$ that lie on different boundary components of $\p$, there is at most one oriented seam edge $y$ such that we can concatenate them into an oriented arc $x \circ y \circ x'$. 

We assume that $w$ and $w'$ come from closed geodesics $\gamma$ and $\gamma'$. If $x$ is the last boundary edge in the subword $b_i$ and $x'$ is the first boundary edge in the subword $b_{i+1}$ then the existence of $p(\gamma)$ implies that there does exist a seam edge $y$ such that $x \circ y \circ x'$ is an oriented arc. Thus $s_i = s'_i$ for each $i$. (Note that we number the boundary subwords modulo $n$, so $b_{n+1} =b_1$.)
\end{proof}

Lastly, we note that because each $\gamma \in \G^c$ is primitive, so is the word $w(\gamma)$.
\begin{lem}
\label{lem:PrimitiveWords}
 Let $w' = w(\gamma')$. Suppose there exists a word $w$ in the edges in $\E$ such that 
 \[
  w' = w^n
 \]
for $n >1$. Then $\gamma'$ is not primitive.
\end{lem}

\begin{proof}
Suppose $w = b_1 s_1 \dots b_n s_n$ where all the subwords except for possibly $s_n$ are non-empty. Then the fact that $w' = b_1 s_1 \dots b_n s_n b_1 \dots$ implies that the concatenation $b_n \circ s_n \circ b_1$ corresponds to an oriented path in the edges of $\E$. Thus we can concatenate the edges in $w$ into a closed path $p$. If $p'$ is the closed curve corresponding to $w'$, then $p' = p^n$. Every closed curve has a unique closed geodesic in its free homotopy class. Let $\gamma$ be the geodesic in the free homotopy class of $p$. Then $p' = p^n$ implies that $\gamma' = \gamma^n$. Therefore $\gamma'$ is not 
primitive.
\end{proof}

\subsection{Word and Geodesic Lengths}
\label{sec:WordAndGeodesicLengths}

The word $w(\gamma)$ encodes geometric properties of $\gamma$ for each $\gamma \in \G^c$. For example, we get the following relationship between the length of a closed geodesic $\gamma$ and the word length of $w(\gamma)$.

\begin{lem}
\label{lem:GeodesicAndWordLength}
Let $\gamma \in \G^c$. If $|w(\gamma)|$ is the word length of $w(\gamma)$, then 
\[
\frac {1}{3} l_{min} |w(\gamma)| \leq l(\gamma) \leq l_{max} |w(\gamma)|
\]
where $l_{min}$ is the length of the shortest boundary edge, and $l_{max}$ is the length of the longest boundary or seam edge in $\E$.
\end{lem}

\begin{proof}
Let $\gamma \in \G^c$ and let $w(\gamma) \in \W$ be the associated word. Let $p(\gamma)$ be the closed curve corresponding to $w(\gamma)$. Throughout this proof, we will use that the number of edges in $p(\gamma)$ is exactly the word length of $w(\gamma)$.

To get the upper bound, we use that $l(\gamma) \leq l(p(\gamma))$ since a geodesic is the shortest curve in its free homotopy class. Thus, if $l_{max}$ is the length of the longest edge in $\E$, then
\[
 l(\gamma) \leq l_{max} |w(\gamma)|
\]

To get the lower bound, set $n$ to be the number of segments in $\gamma$. We first compare $n$ to $|w(\gamma)|$. Let $m$ to be the number of boundary edges in $p(\gamma)$. Note that $m \leq 2n$. To see this, look at the construction of $p(\gamma)$ in the proof of Lemma \ref{lem:ProjectionConstruction}. Each segment $\sigma$ of $\gamma$ got projected onto a single boundary edge, which then may have been replaced by two boundary edges when we did Move 2. Thus, the boundary edges in $p(\gamma)$ account for at most two times the number of segments in $\gamma$.

Two seam edges are never concatenated together and boundary edges appear in consecutive pairs. Thus, as least 2/3 of all edges in $p(\gamma)$ are boundary edges. In other words, $\frac 23 |w(\gamma)| \leq m$. Therefore,
\[
 \frac 2 6 |w(\gamma)| \leq \frac 12 m \leq n
\]
where $\gamma$ has $n$ segments and $p(\gamma)$ has $m$ boundary edges. 

Suppose a segment $\sigma$ has endpoints on seam edges $y$ and $y'$. Because we broke $\p$ up into right angle hexagons, $y$ and $y'$ meet a common boundary edge $x$ at right angles. By some hyperbolic geometry, any arc connecting $y$ and $y'$ will thus be at least as long as $x$, i.e. $l(\sigma) \geq l(x)$. Thus, if $l_{min}$ is the length of the smallest boundary edge in $\E$, then $l_{min} \leq l(\sigma)$ for each segment $\sigma$. Therefore, $l_{min}n \leq l(\gamma)$. So we get the lower bound:
\[
 \frac 13 l_{min} |w(\gamma)| \leq l(\gamma)
\]
\end{proof}

\subsection{An intersection number for words}
We want a notion of word self-intersection number so that if $i(w(\gamma),w(\gamma)) \leq K$, then $i(\gamma, \gamma) \leq K$. 
\begin{defi}
 Let $b_i$ be a boundary subword of $w \in \W$. Suppose $\beta$ is a boundary component of $\p$. We write $b_i \subset \beta$ and say that \textbf{$b_i$ lies on $\beta$} if the boundary edges in $b_i$ lie on $\beta$.
\end{defi}

We define the self-intersection number of a word as follows.
\begin{defi}
\label{def:IntForWords}
Let $w = b_1s_1 \dots b_ns_n \in \W$. Let $\beta_1, \beta_2$ and $\beta_3$ be the boundary components of $\p$. Suppose $w$ has $n_j$ boundary subwords lying on $\beta_j$, $j = 1,2,3$. Let $\sigma_j : \{1, \dots, n_j\} \rightarrow \{1, \dots, n\}$ so that 
\begin{itemize}
 \item $b_{\sigma_j(1)}, \dots, b_{\sigma_j(n_j)} \subset \beta_j$
 \item $|b_{\sigma_j(1)}| \geq |b_{\sigma_j(2)}| \geq \dots \geq |b_{\sigma_j(n_j)}|$
\end{itemize}
Then, let
 \[
  i(w,w) = 2 \sum_{j = 1,2,3} \sum_{i = 1}^{n_j} i|b_{\sigma_j(i)}|
 \]
\end{defi}

In other words, given a word $w = b_1 s_1 \dots b_n s_n$, we group the boundary subwords $b_1, \dots, b_n$ according to the component of $\partial \p$ on which they lie, and then we order the boundary subwords in each group from largest to smallest word length. This gives us the re-indexing functions $\sigma_j$, $ j = 1,2,3$. Then we form the sum above.

\begin{example}[Computing word self-intersection number]
 Suppose
\[
 w = b_1 s_1 \dots b_5 s_5 \in \W
\]
is a word with $b_1, b_3, b_5 \subset \beta_1$ and $b_2, b_4 \subset \beta_2$. Suppose further that $|b_3| \geq |b_1| \geq |b_5|$ and $|b_2| \geq |b_4|$. Then
\[
 i(w,w) = 2 \Big (|b_3| + 2 |b_1| + 3 | b_5|\Big ) + 2\Big (|b_2| + 2 | b_4|\Big )
\]
\end{example}

 \begin{lem}
 \label{lem:IntersectionLowerBound}
  Suppose $w \in \W$ corresponds to the geodesic $\gamma \in \G^c$. Then
  \[
   i(\gamma, \gamma) \leq i(w,w)
  \]
\end{lem}
\begin{proof}
Let $w = b_1s_1 \dots b_n s_n$ be a cyclic word that corresponds to some closed geodesic $\gamma$. We will show how to construct a closed curve $\delta$ freely homotopic to $\gamma$ where
 \[
  |\delta \cap \delta | \leq i (w,w)
 \]

Let $\beta_1, \beta_2, \beta_3$ be the boundary components of $\p$. Suppose $w$ has $n_j$ boundary subwords lying on $\beta_j$, $j = 1,2,3$. Let $\sigma_j : \{1, \dots, n_j\} \rightarrow \{1, \dots, n\}$ so that 
\begin{itemize}
 \item $b_{\sigma_j(1)}, \dots, b_{\sigma_j(n_j)} \subset \beta_j$
 \item $|b_{\sigma_j(1)}| \geq |b_{\sigma_j(2)}| \geq \dots \geq |b_{\sigma_j(n_j)}|$
\end{itemize}
These are the reordering functions from Definition \ref{def:IntForWords}.

First, we construct a region of $\p$ homotopic to the one skeleton of its hexagon decomposition (Figure \ref{fig:SideEdgeRegions}). Let $R_1, R_2$ and $R_3$ be disjoint neighborhoods of the three seam edges. For each boundary subword $b_{\sigma_j(i)}$ lying on $\beta_j$, let $C_{i}^j$ be a cylinder embedded in $\p$ that is homotopic to $\beta_j$. Choose these cylinders so that any two cylinders $C_i^j$ and $C_k^l$ are pairwise disjoint. Lastly, choose $C_i^j$ to be closer to $\beta_j$ than $C_{i+1}^j$ for each $i,j$. The union of $R_1 \cup R_2 \cup R_3$ and the cylinders is homotopic in $\p$ to the union of hexagon edges. (See Figure \ref{fig:SideEdgeRegions}).

We now construct $\delta$ so that
\[
 \delta \subset \bigcup_{k=1,2,3} R_k \bigcup_{i,j}C_i^j
\]
 
\begin{figure}[h!] \centering
 \includegraphics{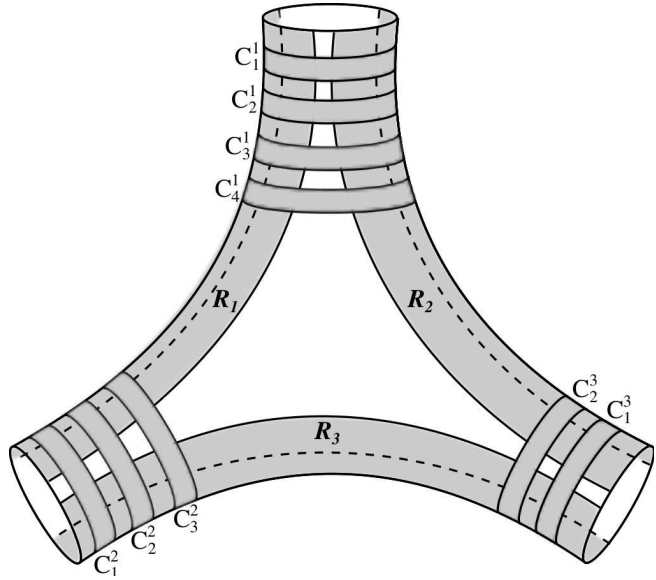}
 \caption[The curve $\delta$ will lie in the gray regions]{The curve $\delta$ will lie in the gray regions.}
 \label{fig:SideEdgeRegions}
\end{figure}

Given a cylinder $C_i^j$, we say its top is the boundary component closest to $\beta_j$ and its bottom is the other boundary component. If $b_{\sigma_j(i)} s_{\sigma_j(i)} b_{\sigma_l(k)}$ is a subword of $w$, we draw a line $l_i^j$ from the top of $C_i^j$ to the bottom of $C_k^l$. We draw $l_i^j$ inside the unique region $R_m, m \in \{1,2,3\}$ that connects the two cylinders. We require that all of the lines in the set $\{l_i^j \ | \ j = 1,2,3, i = 1, \dots, n_j\}$ be pairwise disjoint.

Let $p_i^j$ be the endpoint of $l_i^j$ on cylinder $C_i^j$ and let $q_k^l$ be the endpoint of $l_i^j$ on cylinder $C_k^l$. Since $\gamma$ is closed, each cylinder $C_i^j$ is now decorated with a point $p_i^j$ on its top boundary and a point $q_i^j$ on its bottom boundary. The boundary subword $b_{\sigma_j(i)}^j$ determines a twisting direction about $\beta_j$. So in each cylinder $C_{i}^j$, we draw a curve from $p_i^j$ to $q_i^j$ that twists in this direction for $|b_{\sigma_j(i)}|$ half-twists. Call this curve $\delta_i^j$. 

There is just one natural way to concatenate the twisting arcs $\delta_i^j$ with the lines $l_i^j$ to form a closed curve. Call this concatenation $\delta$. So,
\[
 \delta = \mathbf{O}_{i,j = 1}^n \ \delta_i^j \ \mathbf{O}_{i,j = 1}^n \ l_i^j 
\]
where we take the concatenations in the order that makes sense (Figure \ref{fig:CurveInRegions}). 
\begin{figure}[h!] \centering
 \includegraphics{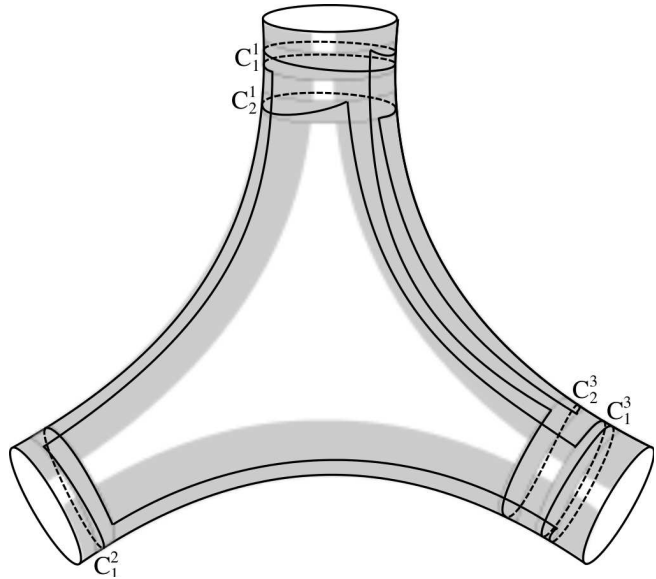}
 \caption[An example of a curve corresponding to a word $w$]{An example of a curve $\delta$ constructed from a word $w$.}
 \label{fig:CurveInRegions}
\end{figure}

Now we just need to count the self-intersections of $\delta$. Since $R_1,R_2$ and $R_3$ are pairwise disjoint, and since any two pairs of cylinders are also disjoint, the only intersections occur when a curve $\delta_i^j$ in a cylinder $C_i^j$ intersects a line $l_k^l$ in a region $R_k$. If $l_k^l$ passes through $C_i^j$, and if $\delta_i^j$ has $|b_{\sigma_j(i)}|$ half-twists, then 
\[
 |l_k^l \cap \delta_i^j| \leq \frac{|b_{\sigma_j(i)}|}{2} + 1
\]

By construction, the cylinder $C_i^j$ lies between cylinders $C_1^j, \dots, C_{i-1}^j$ and the rest of $\p$. So, the curve $\delta_i^j$ is intersected only by lines with endpoints on the cylinders $C_1^j, \dots, C_i^j$. In particular, both lines coming out of $C_1^j, \dots, C_{i-1}^j$ cross $\delta_i^j$. However, $\delta_i^j$ is only intersected by the line with endpoint on the top boundary component of $C_i^j$, not the line whose endpoint is on the bottom boundary. Thus, each cylinder $C_i^j$ contains at most $(2i-1)(\frac 12 |b_{\sigma_j(i)}| + 1)$ unique intersection points of $\delta \cap \delta$. Therefore,
\[
 |\delta \cap \delta| \leq \sum_{j =1,2,3} \sum_{i=1}^{n_j} i(|b_{\sigma_j(i)}| + 2)
\]
since $(2i-1)|(\frac 12 b_{\sigma_j(i)}| + 1) = (|b_{\sigma_j(i)}| + 2) (i -1)$ and $ i-1 \leq i$.

Since geodesics have the least number of self-intersections in their free homotopy class, this implies that
\[
 i(\gamma ,\gamma) \leq \sum_{j =1,2,3} \sum_{i=1}^{n_j} i(|b_{\sigma_j(i)}| + 2)
\]
Lastly, since $|b_i| \geq 2$ for each $i = 1, \dots, n$, we have that $|b_{\sigma_j(i)}| + 2 \leq 2|b_{\sigma_j(i)}|$. Therefore, we arrive at
\[
 i(\gamma ,\gamma) \leq \sum_{j =1,2,3} \sum_{i=1}^{n_j} 2i|b_{\sigma_j(i)}|
\]
where the right-hand side is exactly the self-intersection number for words defined above.
\end{proof}

\section[Constructing geodesics in $\G^c(L,K)$]{Constructing geodesics in $\G^c(L,K)$ to get a lower bound}
\label{sec:ConstructingGeodesics}

In this section, we prove the lower bound on $\#\G^c(L,K)$ for a pair of pants. Let us restate it here:
\begingroup

\def \thetheorem{\ref{thm:LowerBound}}
\begin{theorem} Let $\p$ be a hyperbolic pair of pants.  Let $l_{max}$ be the longest edge in the hexagon decomposition of $\p$. If $L \geq 8 l_{max}$ and $K \geq 12$, we have that
\[
 \# \G^c(L,K) \geq 2 + \frac{1}{2} \min\{  2^{\frac{1}{8l_{max}}L},2^{\sqrt{\frac{ K}{ 12}}}\}
\]
\end{theorem}
\addtocounter{theorem}{-1}
\endgroup

In the version of Theorem \ref{thm:LowerBound} from the introduction, we define $l_{max}$ to be the longer of the length of the longest boundary component of $\p$ or distance between boundary components of $\p$. This allowed us to define the $l_{max}$ without referring to a hexagonal decomposition of $\p$. That definition of $l_{max}$ is at most twice the $l_{max}$ defined in this formulation. In fact, the length of each boundary edge is half the length of the boundary component on which it lies. Also, the length of each seam edge is exactly the distance between the boundary components it connects. So when we prove this version of Theorem \ref{thm:LowerBound}, we also prove the version stated in the introduction.

\subsection{Proof summary}
\label{sec:ProofSummary}
The proof of the theorem is organized as follows.
\begin{itemize}
 \item We have a canonical surjection $\W \rightarrow \G^c$. The map $\G^c \rightarrow \W$ described in Section \ref{sec:CombModel} is a section of this map. However, there is no easy way to describe the image of this section. So instead, we
%
%
%
%
 describe a large subset $\W_\Gamma \subset \W$ so that the map $\W_\Gamma \rightarrow \G^c$ is injective. This is done as follows:
 \begin{itemize}
  \item In Lemma \ref{lem:GraphWordConstruction}, we give a construction that turns closed paths in the graph $\Gamma_\E$, found in Figure \ref{fig:EdgeGraph}, into words that lie in $\W$. This means we get a map from closed paths $\tau$ in $\Gamma_\E$ to closed geodesics $\gamma(\tau) \in \G^c$:
 \[
  \tau \rightarrow w(\tau) \rightarrow \gamma(\tau)
 \] 
 \item In Lemma  \ref{lem:GraphConstrInjective}, we show that this map is one-to-one. So we get an injection from words in $\W$ that come from closed paths in $\Gamma_\E$ to closed geodesics in $\G^c$.
  \end{itemize}
    \end{itemize}
If $\W_\Gamma(L,K) \subset \W_\Gamma$ is the set of words that map to $\G^c(L,K)$, we get a lower bound on $\#\W_\Gamma(L,K)$ as follows:
 \begin{itemize}
 \item We get conditions on closed paths $\tau$ in $\Gamma_\E$ so that $\gamma(\tau) \in \G^c(L,K)$. In fact, we find a function $N(L,K)$ so that if $\tau$ is a closed path in $\Gamma_\E$ and $\gamma(\tau)$ is the corresponding closed geodesic then
 \[
  |\tau| \leq N(L,K) \implies \gamma(\tau) \in \G^c(L,K)
 \]
where $|\tau|$ is the path length of $\tau$ (Lemma \ref{lem:BoundOnPathLength}.)

To find $N(L,K)$, we use the map $\tau \mapsto w(\tau) \in \W$ as an intermediary. The construction of $w(\tau)$ directly implies that 
\[
|w(\tau)| \leq 4 |\tau| \mbox{ and } i(w,w) \leq  3 |\tau|^2
\]
We then apply Lemmas \ref{lem:GeodesicAndWordLength} and \ref{lem:IntersectionLowerBound} that say
\[
l(\gamma) \leq l_{max}|w(\gamma)| \mbox{ and } i(\gamma, \gamma) \leq i(w(\gamma), w(\gamma))
\]
to get the relationship between $|\tau|$ and the quantities $l(\gamma(\tau))$ and $i(\gamma(\tau), \gamma(\tau))$.
 
 
 \item We get an explicit formula for the number of closed paths in $\Gamma_\E$ of length $N$ (Lemma \ref{lem:NumberOfPaths}).
 
 \item Estimating the number of paths of length at most $N(L,K)$ gives a lower bound on $\#\G^c(L,K)$ for a pair of pants (Section \ref{sec:PfLowerBound}.)
\end{itemize}

\subsection{Building words that correspond to closed geodesics}

Consider the labeling of the (oriented) boundary edges in $\E$ given in Figure \ref{fig:PantsBoundaryEdges}. If $x_i$ is a labeled edge, then the same edge with the opposite orientation will be denoted $x_i^{-1}$.

\begin{figure}[h!] \centering
 \includegraphics{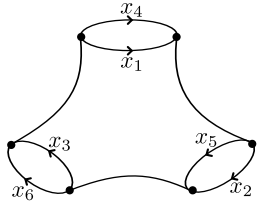}
 \caption{Labeled boundary edges}
 \label{fig:PantsBoundaryEdges}
\end{figure}

\begin{figure}[h!] \centering
 \includegraphics{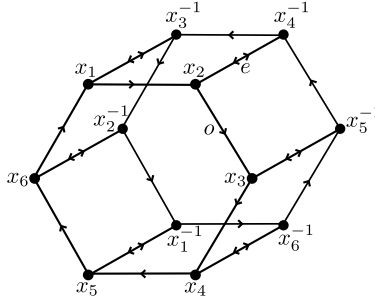}
 \caption[The graph $\Gamma_\E$]{The graph $\Gamma_\E$. All one-way edges are labeled $o$ and all two-way edges are labeled $e$.}
 \label{fig:EdgeGraph}
\end{figure}

\begin{defi}
 We say a path is cyclic if it is a cyclic word in its vertices. It is primitive if the word is primitive.
\end{defi}

\begin{lem}
\label{lem:GraphWordConstruction}
Any cylic path $\tau$ in the directed graph $\Gamma_\E$ in Figure \ref{fig:EdgeGraph} corresponds to a cyclic word $w(\tau) \in \W$ and an oriented closed geodesic $\gamma(\tau)$.
\end{lem}

\begin{example}
\label{ex:GammaPath}
 The cyclic path $x_1 x_2 x_4^{-1} x_3^{-1}$ (the top square) corresponds to the cyclic word 
 \[
 w =  (x_1 x_4^{-1} x_1) \cdot s_1 \cdot ( x_2 x_5^{-1}) \cdot s_2 \cdot (x_4^{-1} x_1 x_4^{-1}) \cdot s_3 \cdot (x_3^{-1} x_6) \cdot s_4
  \]
where $s_1, \dots, s_4$ are the unique side edges that make $w$ a word in $\W$.
\end{example}

\begin{proof}
We start with a description of $\Gamma_\E$. The vertices of $\Gamma_\E$ are the edges in $\E$ as labeled in Figure \ref{fig:PantsBoundaryEdges}. The vertices on the front hexagon are labeled by edges $x_1, \dots, x_6$ and the vertices on the back hexagon are labeled by edges $x_1^{-1}, \dots,x_6^{-1}$. For each $i$, a directed edge labeled $o$ goes from vertex $x_i$ to $x_{i+1}$ and from $x_i^{-1}$ to $x_{i-1}^{-1}$. A bidirectional edge labeled $e$ goes between vertices $x_i$ and $x_{i+2}^{-1}$. (All indices are taken modulo 6.)

Take a cyclic path $\tau$ in $\Gamma_\E$ with $\tau = v_1 \dots v_n$ (where $v_i$ is a vertex of $\Gamma_\E$ for each $i$.) We associate to each vertex $v_i$ a boundary subword $b_i$. The first letter of $b_i$ is the edge label of $v_i$. If $v_i$ is joined to $v_{i+1}$ by an edge labeled $o$, then $|b_i|=3$. Otherwise $|b_i| = 2$. Note that specifying the initial letter and length of the boundary subword uniquely determines $b_i$.

In Example \ref{ex:GammaPath}, $v_1$ has label $x_1$ and $x_1$ is joined to $x_2$ by an edge labeled $o$. So $b_1 = x_1 x_4^{-1} x_1$ has length 3 and starts with $x_1$.

We claim that there are seam edges $s_1, \dots, s_n$ so that $b_1s_1 \dots b_n s_n \in \W$. There are four cases to check: $v_i$ could have the form $x_j$ or $x_j^{-1}$ and the edge from $v_i$ to $v_{i+1}$ could be labeled either $e$ or $o$. Suppose $v_i$ is labeled $x_j$ and the edge from $v_i$ to $v_{i+1}$ is labeled e, so $|b_i| = 2$. Then given the edge labels in Figure \ref{fig:PantsBoundaryEdges}, $b_i = x_j x_{j+3}^{-1}$. 

The edge labeled $e$ joins vertex $x_j$ to vertex $x_{j+2}^{-1}$. Thus, $v_{i+1} =  x_{j+2}^{-1}$. So $b_{i+1}$ starts with $x_{j+2}^{-1}$. There is a seam edge between $x_{j+3}^{-1}$ and $x_{j+2}^{-1}$ for all $j = 1, \dots, 6$. So there is a seam edge $s_i$ so that $b_i s_i b_{i+1}$ forms a non-backtracking path. The other cases can be checked in the same way. 

Therefore, the cyclic path $\tau = v_1 \dots v_n$ corresponds to a cyclic word $w(\tau) = b_1s_1 \dots b_ns_n \in \W$. Since each word in $\W$ corresponds to a closed geodesic, $\tau$ corresponds to a closed geodesic $\gamma(\tau) \in \G^c$.

\end{proof}

\subsection{Map from paths in $\Gamma_\E$ to geodesics injective}

We now have maps
\[
 \tau \rightarrow w(\tau) \rightarrow \gamma(\tau)
\]
We show that $w(\tau)$ actually lies in the following special class of words.

\begin{defi}
 A cyclic word $b_1s_1 \dots b_ns_n \in \W$ is \textbf{alternating} if no two consecutive boundary edges lie on the same hexagon. In particular, the last edge of $b_i$ does not lie in the same hexagon as the first edge of $b_{i+1}$ for each $i$.
\end{defi}

\begin{claim}
 For each cyclic path $\tau$ in $\Gamma_\E$, the word $w(\tau)$ constructed in the proof of Lemma \ref{lem:GraphWordConstruction} is a cyclic alternating word.
\end{claim}
\begin{proof}
 Let $\tau = v_1 v_2 \dots v_n$ be a cyclic path in $\Gamma_\E$. Then $\tau$ corresponds to a word $b_1 s_1 \dots b_n s_n \in \W$. Inside each boundary subword, adjacent boundary edges lie on different hexagons. So we just need to check that the last boundary edge of $b_i$ lies on a different hexagon than the first boundary edge in $b_{i+1}$. Once again, this can be done by considering four cases. We have the cases where $v_i$ is of the form $x_j$ or the form $x_j^{-1}$ and the cases where $v_i$ is joined to $v_{i+1}$ by an edge labeled $e$ or an edge labeled $o$.
 
 We will do the case where $v_i$ is labeled $x_j$ and $v_i$ is joined to $v_{i+1}$ by an edge labeled $e$. If $v_i$ is labeled $x_j$ and the edge is labeled $e$, then $b_i = x_j x_{j+3}^{-1}$. Since $v_i$ is joined to $v_{i+1}$ by an edge labeled $e$, we have that $b_{i+1}$ starts with the letter $x_{j+2}^{-1}$. We see from Figure \ref{fig:PantsBoundaryEdges} that $x_{j+3}^{-1}$ and $x_{j+2}^{-1}$ lie on different hexagons. The other 3 cases are shown in the same way. 
\end{proof}


\begin{lem}
\label{lem:GraphConstrInjective}
 Suppose $\tau \neq \tau'$ are two primitive cyclic paths in $\Gamma_\E$. Then $\gamma(\tau) \neq \gamma(\tau')$.
\end{lem}
\begin{proof}
 Let $w(\tau)$ and $w(\tau')$ be the words constructed from $\tau$ and $\tau'$, respectively. We know that $w(\tau)$ and $w(\tau')$ are cyclic alternating words. Because $\tau$ and $\tau'$ are primitive, $w(\tau)$ and $w(\tau')$ are primitive as well. We will show that if $w$ and $w'$ are primitive, cyclic alternating words, then their geodesics $\gamma(w)$ and $\gamma(w')$ are distinct.
 
 Let $w$ be a primitive, cyclic alternating word and let $\gamma(w)$ be the corresponding geodesic. Let $p(w)$ be the closed curve in $\p$ formed by concatenating the edges in $w$. 
 Lift $p(w)$ to a complete curve $\tilde p(w)$ in the universal cover, $\tilde \p$.
 
 The hexagon decomposition of $\p$ lifts to a hexagonal tiling of $\tilde \p$. We get a graph $\Gamma$ dual to this hexagonal tiling: Put a vertex in the middle of each hexagon, and join two vertices if their hexagons share a side edge. This graph is a valence 3 tree. (See Figure \ref{fig:HexTilingDualGraph}).
 \begin{figure}[h!]
  \includegraphics{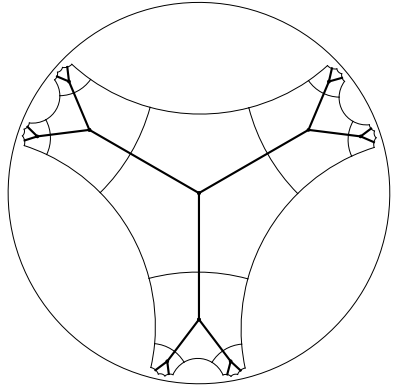}
  \caption{A piece of the graph $\Gamma$ dual to the hexagonal tiling of $\tilde \p$}
  \label{fig:HexTilingDualGraph}
 \end{figure}

 
Let $\Gamma(w)$ be the subgraph of $\Gamma$ that has a vertex for every hexagon that contains a boundary edge of $\tilde p(w)$ (see Figure \ref{fig:ThreeWs}). We want to show that $\Gamma(w)$ is an embedded line. If not, then $\Gamma(w)$ would have a valence 1 vertex. This would correspond to $\tilde p(w)$ entering a hexagon $h$, traversing some of its boundary edges, and then leaving $h$ through the same seam edge through which it entered. But this cannot be achieved if $\tilde p(w)$ never has more than one consecutive boundary edge in the same hexagon. Thus $\Gamma(w)$ is an embedded line.
 
 \begin{figure}[h!]
  \includegraphics{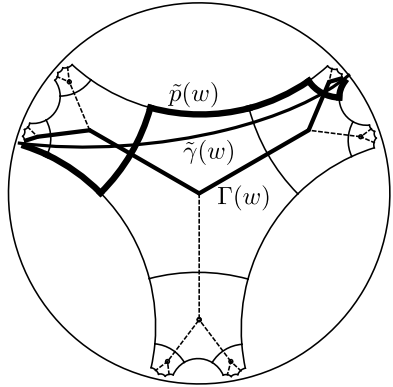}
  \caption{The subgraph $\Gamma(w)$ goes through the same hexagons as $\tilde \gamma(w)$.}
  \label{fig:ThreeWs}
 \end{figure}
 

 

We can do the same construction for any other primitive, cyclic alternating word $w'$. Let $\gamma(w')$ be the geodesic corresponding to $w'$ and let $p(w')$ be the concatenation of edges in $w'$. Lift $p(w')$ to a curve $\tilde p(w')$ and construct the subgraph $\Gamma(w')$. This subgraph is again a line embedded in $\Gamma$. 

Note that complete geodesics in $\tilde \p$ are in one-to-one correspondence with embedded lines in $\Gamma$. Therefore, $\Gamma(w)$ corresponds to a unique complete geodesic $\tilde \gamma$ that must be a lift of $\gamma(w)$, and likewise, $\Gamma(w')$ corresponds to a unique complete geodesic $\tilde \gamma'$ that must be a lift of $\gamma(w')$ (see Figure \ref{fig:ThreeWs}).

Suppose for contradiction that $\gamma(w) = \gamma(w')$ as oriented geodesics. So we could have chosen a lift $\tilde p(w')$ so that $\Gamma(w) = \Gamma(w')$ as oriented paths in $\Gamma$. (Note that the orientations on $\Gamma(w)$ and $\Gamma(w')$ come from the orientations of $p(w)$ and $p(w')$, respectively.) The vertices of $\Gamma(w)$ are in one-to-one correspondence with the hexagons traversed by $\tilde p(w)$, and the analogous statement is true for $\tilde p(w')$. So $\tilde p(w)$ and $\tilde p(w')$ pass through the exact same hexagons in $\tilde \p$.

Suppose $\Gamma(w)$ and $\Gamma(w')$ pass through consecutive hexagons $h_1, h_2$ and $h_3$. Then $\tilde p(w)$ and $\tilde p(w')$ both travel from $h_1$ to $h_3$ through the boundaries of these hexagons. Furthermore, they each pass through just one boundary edge in $h_2$. Therefore, they both pass through the same boundary edge of $h_2$ (see Figure \ref{fig:ThreeHex}.) Thus, $\tilde p(w)$ and $\tilde p(w')$ pass through all the same boundary edges. But there is just one side edge that can lie between a pair of boundary edges. So $\tilde p(w)$ and $\tilde p(w')$ must have the same image in $\p$. Since $w$ and $w'$ are primitive, this implies $p(w) = p(w')$ as cyclic paths. So $w = w'$ as cyclic words.

\begin{figure}[h!]
 \includegraphics{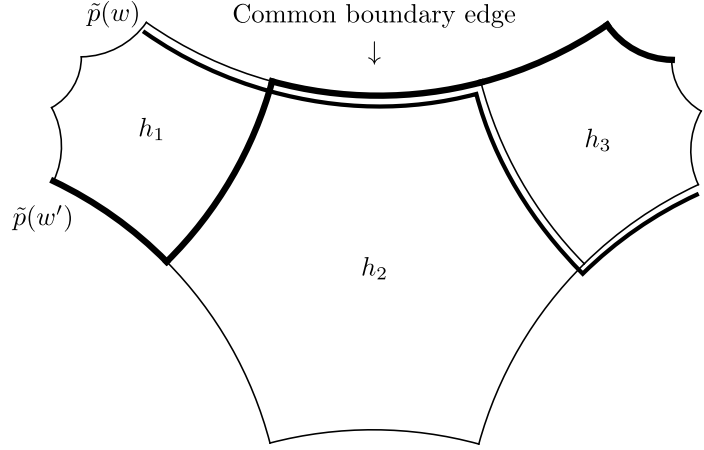}
 \caption{If $\tilde p(w)$ and $\tilde p(w')$ pass through $h_1, h_2$ and $h_3$, and $w$ and $w'$ are cyclic alternating words, then $\tilde p(w)$ and $\tilde p(w')$ must pass  through the same boundary edge of $h_2$.}
 \label{fig:ThreeHex}
\end{figure}

We showed that if $\gamma(\tau) = \gamma(\tau')$ as oriented geodesics then $w(\tau) = w(\tau')$ as cyclic words. Since $w(\tau)$ always has more than one boundary subword, we can recover $\tau$ from $w(\tau)$. In other words, the map $\tau \rightarrow w(\tau)$ is injective. So $\tau = \tau'$ as cyclic paths. Therefore two primitive, cyclic paths $\tau$ and $\tau'$ in $\Gamma_\E$ are equal if and only if $\gamma(\tau) = \gamma(\tau')$ as oriented geodesics.

\end{proof}

\subsection{Path length and geodesic length and intersection number}

Fix $L$ and $K$. We want to count the number of cyclic paths $\tau$ in $\Gamma_\E$ so that the corresponding closed geodesic satisfies $\gamma(\tau) \in \G^c(L,K)$. First, we show that we can guarantee $\gamma(\tau) \in \G^c(L,K)$ with just an upper bound on path length $|\tau|$.

\begin{lem}
\label{lem:BoundOnPathLength}
 Fix $L$ and $K$. Suppose we have a cyclic path $\tau$ in $\Gamma_\E$ of length at most $N(L,K)$, for
 \[
N(L,K) = \min\{\frac{1}{4 l_{max}} L, \sqrt{\frac{1}{3} K} \}
 \]
Then $\gamma(\tau) \in \G^c(L,K)$, where $\gamma(\tau)$ is the closed geodesic corresponding to $\tau$.
\end{lem}
\begin{proof}
 Take a cyclic path $\tau$ in $\Gamma_\E$ of length at most $N(L,K)$. Then $\tau$ corresponds to a word $w(\tau) = b_1s_1 \dots b_ns_n$. By construction,
 \[
  n = |\tau|
 \]
where $|\tau|$ denotes the path length of $\tau$.

We constructed $w(\tau)$ so that $|b_i| \leq 3$ for each $i$. Thus, $|b_i s_i| \leq 4, \forall i$. Therefore, 
\[
 |w(\tau)| \leq 4 |\tau|
\]

Furthermore, we can get a bound on $i(w(\tau),w(\tau))$ as follows. We have that
\[
i(w(\tau),w(\tau)) = 2 \sum_{j = 1,2,3} \sum_{i = 1}^{n_j} i|b_{\sigma_j(i)}|
\]
for appropriate indices $\sigma_1(i), \sigma_2(i)$ and $\sigma_3(i)$ and where $n_j$ is the number of $b_i$ that lie on boundary component $\beta_j$ of $\p$. Thus,
\begin{align*}
 i(w(\tau),w(\tau)) & \leq 2 \sum_{j = 1,2,3} \sum_{i = 1}^{n_j} 3 i \\
  & = 3[(n_1^2 + n_1) + (n_2^2 + n_2)+(n_3^2 + n_3)]\\
  & \leq 3 (n_1 + n_2 + n_3)^2\\
  & = 3n^2 \\
  & = 3 | \tau|^2
\end{align*}
Note that we get the second inequality because $n_1, n_2$ and $n_3$ are non-negative integers. Therefore,
\[
 i(w(\tau), w(\tau)) \leq 3 |\tau|^2
\]

 Let $\gamma(\tau)$ be the geodesic corresponding to $\tau$ (and $w(\tau)$). By Lemmas \ref{lem:GeodesicAndWordLength} and \ref{lem:IntersectionLowerBound},
 \[
  l(\gamma(\tau)) \leq l_{max} |w(\tau)| \mbox{ and } i(\gamma(\tau), \gamma(\tau)) \leq i(w(\tau), w(\tau))
 \]
 where $l_{max}$ is the length of the longest boundary or seam edge in $\E$. Therefore,
 \[
   l(\gamma(\tau)) \leq 4 l_{max} |\tau|
 \]
and 
\[
 i(\gamma(\tau), \gamma(\tau)) \leq 3 |\tau|^2
\]

In particular, if
\[
 |\tau| \leq \frac{1}{4 l_{max}} L \mbox{ and } |\tau| \leq \sqrt{\frac{K}{3}}
\]
 then $\gamma(\tau) \in \G^c(L,K)$.
 
\end{proof}

\subsection{Counting paths}

By Lemma \ref{lem:BoundOnPathLength}, we can get a lower bound on $\#\G^c(L,K)$ via a lower bound on the number of paths of length $N(L,K)$. In the following lemma, we count the number of cyclic paths of length exactly $n$.

\begin{lem}
\label{lem:NumberOfPaths}
 Let $H_n$ be the number of cyclic paths in $\Gamma_\E$ of length exactly $n$. If $n$ is odd, then $H_n = 0$. Otherwise,
 \[
  H_{2n} = 2 + \frac 1n \sum_{d | n} \phi(d) 4^{n/d}
 \]
where $\phi(d)$ is the Euler totient function that counts the number of $k \leq d$ relatively prime to $d$.
\end{lem}
This closed form for $H_n$ was communicated to us by Alex Miller \cite{Miller}. It follows from combining the eigenvalues of the edge adjacency matrix $M$ for the graph $\Gamma_\E$ with Burnside's lemma from group theory.

\begin{proof}
Take $\Gamma_\E$ and relabel the vertices $1, \dots, 12$ so that $x_1 \dots, x_6$ are relabeled $1, \dots , 6$, respectively and $x_1^{-1}, \dots, x_6^{-1}$ are relabeled $7, \dots, 12$, respectively. Form its $12 \times 12$ edge adjacency matrix $M$. This is the matrix whose $(i,j)$ entry is 1 if there is an edge from vertex $i$ to vertex $j$, and 0 otherwise. So,
 \begin{align*}
  M & = \left [
 \begin{array}{cccccc  cccccc}
    & 1 &   &   &   &   &      &   & 1 &   &   &   \\
    &   & 1 &   &   &   &      &   &   & 1 &   &   \\
    &   &   & 1 &   &   &      &   &   &   & 1 &   \\
    &   &   &   & 1 &   &      &   &   &   &   & 1 \\
    &   &   &   &   & 1 &    1 &   &   &   &   &   \\
  1 &   &   &   &   &   &      & 1 &   &   &   &   \\ 
    &   &   &   & 1 &   &      &   &   &   &   & 1 \\ 
    &   &   &   &   & 1 &    1 &   &   &   &   &   \\
  1 &   &   &   &   &   &      & 1 &   &   &   &   \\
    & 1 &   &   &   &   &      &   & 1 &   &   &   \\
    &   & 1 &   &   &   &      &   &   & 1 &   &   \\
    &   &   & 1 &   &   &      &   &   &   & 1 &   \\ 
 \end{array}
 \right ] \\
 & = \left [
 \begin{array}{cc}
  A & A^2 \\
  A^{-2} & A^{-1}
 \end{array}
\right ]
 \end{align*} 
 where $A$ is a matrix so that $A^6 = I$. 
 
 Suppose we take the $n^{th}$ power $M^n$ of $M$. Then the $(i,j)$ entry of $M^n$ is exactly the number of  paths going through $n + 1$ vertices, that start at vertex $i$ and end at vertex $j$.  
 Let $\Omega_n$ be the set of (non-closed) paths $v_1 v_2 \dots v_n$ in $\Gamma_\E$ so that there is an edge from $v_n$ to $v_1$. Then, 
 \[
  tr(M^n) = \# \Omega_n
 \]
 
 The cyclic paths of length $n$ in $\Gamma_\E$ correspond to exactly the elements of $\Omega_n$ up to cyclic permutation. So if $C_n$ is the cyclic group of order $n$, then $H_n = \# \left ( \Omega_n / C_n \right )$.
 
 The Burnside lemma says that the number of $C_n$-orbits in $\Omega_n$ is the average number of fixed points of the $C_n$ action. So,
 \[
  H_n = \frac 1n \sum_{\sigma \in C_n} \# \{ \omega \in \Omega_n \ | \ \sigma \omega = \omega\}
 \]

 Choose an element $\sigma \in C_n$. If $|\sigma| = d$, then $\sigma$ is the product of $n/d$ disjoint $d$-cycles. If $\omega \in \Omega_n$ is a word so that $\sigma(\omega) = \omega$, then $\omega = \nu^d$, where $\nu \in \Omega_{n/d}$. In other words, $|\sigma| = d$ implies the fixed set of $\sigma$ is in one-to-one correspondence with $\Omega_{n/d}$. The number of order $d$ elements of $C_n$ is $\phi(d)$, where $\phi$ is the Euler totient function. Therefore,
 \[
  H_n = \frac 1n \sum_{d | n} \phi(d) | \Omega_{n/d} |
 \]
 As previously noted, $|\Omega_{n/d}| = tr(M^{n/d})$. So we can write this sum as
 \[
   H_n = \frac 1n \sum_{d | n} \phi(d) tr(M^{n/d})
 \]
 The trace of a matrix is the sum of its eigenvalues. The characteristic polynomial of $M$ can be computed to be $-\lambda^6(\lambda - 2) (\lambda-1)^2(\lambda +1)^2(\lambda+2)$, and so its non-zero eigenvalues are $-2,-1,-1,1,1$, and 2. Therefore, 
 \begin{align*}
  tr(M^{k}) & = (-2)^k + 2(-1)^k + 2^k + 2 \\
   & = (2^k + 2) (1 + (-1)^k)\\
   & = \left \{ 
    \begin{array}{cc}
     0 & k \text{ odd}\\
     2(2^k + 2) & k \text{ even}
    \end{array}
    \right. 
 \end{align*}

 If $n$ is odd, then $n/d$ is odd, and so $tr(M^{n/d}) = 0$ for all $d | n$. Therefore, $H_n = 0$ if $n$ is odd. If $n = 2m$ is even, note that $d | m$ if and only if $n/d$ is even. So, we need only sum over those $d$ that divide $m$:
 \begin{align*}
  H_{2m} & = \frac{1}{2m} \sum_{d | m} \phi(d)2(2^{\frac{2m}{d}} + 2) \\
    & = \frac{1}{m} \sum_{d | m} \phi(d) 2^{\frac{2m}{d}} + 2 \frac{1}{m} \sum_{d | m} \phi(d) \\
    & = 2 + \frac{1}{m} \sum_{d | m} \phi(d) 4^{\frac{m}{d}}
 \end{align*}
where the last equality comes from the fact that $\sum_{d | m} \phi(d) = m$.
\end{proof}


\subsection{Proof of Theorem \ref{thm:LowerBound}.}
\label{sec:PfLowerBound}
We now get a lower bound on $\#\G^c(L,K)$. 

In Lemma \ref{lem:NumberOfPaths}, we show that the number of length $2m$ cyclic paths in $\Gamma_\E$ is 
\begin{align*}
 H_{2m} & = 2 + \frac{1}{m} \sum_{d | m} \phi(d) 4^{\frac{m}{d}} \\
  & \geq 2 + \frac 1 m 4^m\\
  & \geq 2 + 2^m
\end{align*}
where these inequalities hold for all $m > 0$. So we will use the simplified inequality $H_{2m} \geq 2 + 2^m$.

Any cyclic path of length exactly $2m$ can be reduced to a primitive cyclic path of length at most $2m$. So this gives a lower bound on the number of primitive cyclic paths of length at most $2m$. 

By Lemma \ref{lem:BoundOnPathLength}, if $\tau$ is a cyclic path in $\Gamma_\E$ with $|\tau| \leq N(L,K)$ for 
\[
 N(L,K) = \min \{\frac{1}{4 l_{max}} L, \sqrt{\frac K3}\}
\]
then $\gamma(\tau) \in \G^c(L,K)$. So we need to get a lower bound on the number of primitive, cyclic paths in $\Gamma_\E$ of length at most $N(L,K)$.

The function $2 + 2^m$ is increasing in $m$ for all $m$. There is some even number between $N(L,K) -2$ and $N(L,K)$. Assuming $N(L,K) - 2 > 0$, we can set $2m = N(L,K) - 2$, and get that the number of primitive, cyclic paths in $\Gamma_\E$ of length at most $N(L,K)$ is at least
\[
\#\G^c(L,K) \geq  2 + 2^{\frac 12 N(L,K) - 1} = 2+ \frac 12 2^{\frac 12 N(L,K)}
\]

If $N(L,K) = \sqrt{\frac K 3}$, this tells us that $K \geq 12$ implies
\[
 \#\G^c(L,K) \geq 2+ \frac 1 2 2^{\sqrt{ \frac{K}{12}}}
\]
If $N(L,K) = \frac{1}{4 l_{max}} L$, then $L \geq 8 l_{max}$ implies
\[
 \#\G^c(L,K) \geq 2+ \frac 1 2 2^{\frac{1}{8 l_{max}} L}
\]
Therefore, if $L \geq 8 l_{max}$ and $K \geq 12$,
\[
 \# \G^c(L,K) \geq
 2+ \frac 1 2  \min\{
  2^{\sqrt{ \frac{K}{12}}},  2^{\frac{1}{8 l_{max}} L} 
 \}
\]

%

\section{Lower bound for surfaces}
\label{sec:SurfaceBound}

Let $\S$ be an arbitrary surface. The lower bound for $\#\G^c(L,K)$ on $\S$ follows from the lower bound on pairs of pants. The idea is that we will count geodesics in different pairs of pants inside $\S$. To make sure that we do not over-count, we need the following lemma. 
\begin{lem}
\label{lem:CurvesDistinct}
 Let $f_1, f_2: \p \rightarrow \S$ be two embeddings of a pair of pants into $\S$. Suppose $\gamma_1$ and $\gamma_2$ are two non-simple closed curves on $\p$. If $f_1$ is not homotopic to $f_2$, then $f_1(\gamma_1)$ is not freely homotopic to $f_2(\gamma_2)$ inside $\S$.
\end{lem}
\begin{proof}
 
 Suppose $f_1(\gamma_1)$ is freely homotopic to $f_2(\gamma_2)$. Then the geodesic representatives of $f_1(\gamma_1)$ and $f_2(\gamma_2)$ are the same. Let $\phi$ be the geodesic representative of these two curves.
 
 We can tighten the boundary curves of $f_1(\p)$ and $f_2(\p)$ to get pairs of pants $\p_1$ and $\p_2$ with geodesic boundary. Then $\phi$ is a non-simple geodesic that lies in both $\p_1$ and $\p_2$. Using the Euler characteristic, we see that connected components of $\p_1 \cap \p_2$ can be disks, cylinders, or a pair of pants. Since $\phi$ is non-simple, it cannot lie in a disk or a cylinder. So some component of $\p_1 \cap \p_2$ must be a pair of pants. But this means $\p_1 = \p_2$. Since homotopies of simple closed multicurves can be extended to ambient homotopies, $f_1$ is homotopic to $f_2$.
 \end{proof}

We will now prove the main theorem for arbitrary surfaces. We give a more precise, but messier, version of the theorem here.
\begingroup

\def \thetheorem{\ref{thm:LowerBoundSurface}}
\begin{theorem}
 Let $\S$ be a genus $g$ surface with $n$  geodesic boundary components, and let $X$ be a negatively curved metric on $\S$. Then whenever $K > 12$ and $L > 3 s_X \sqrt K$ we have 
 \[
 \# \G^c(L,K) \geq c(X)(\frac{L}{6 \sqrt K})^{6g-6+2n}(2 + \frac 12 2^{\sqrt{\frac{ K}{12}}})
 \]
 where $s_X$ and $c_X$ are constants that depend only on the metric $X$.
\end{theorem}
\addtocounter{theorem}{-1}
\endgroup
\noindent \textbf{NB:} The constant $s_X$ is related to the width of a collar neighborhood of the systole in $X$.
\begin{proof}
 Consider the set of all pairs of pants $\p$ with geodesic boundary components inside $\S$. Given any such $\p \subset \S$, let $l_{max}(\p)$ be the length of the longest boundary component or longest arc connecting boundaries of $\p$ at right angles, inside $\p$. Then by Theorem \ref{thm:LowerBound} and Lemma \ref{lem:CurvesDistinct}
 \begin{equation}
  \label{eq:SumOfPants}
 \#\G^c(L,K) \geq \sum_{\substack{\p \subset \S \\ L \geq 8 l_{max}(\p)}}  2+ \frac 1 2  \min\{
  2^{\sqrt{ \frac{K}{12}}},  2^{\frac{1}{8 l_{max}(\p)} L}\}
 \end{equation}
 
 The condition $L \geq 8 l_{max}(\p)$ on each pair of pants $\p$ in the above sum comes from Theorem \ref{thm:LowerBound}. The other condition (that $K \geq 12$) is already assumed. Furthermore, simple closed curves are not counted by Theorem \ref{thm:LowerBound}. A simple closed curve  would be encoded by a length 1 path in $\Gamma_\E$ that stays at a single vertex, but $\Gamma_\E$ has no edges from a vertex to itself. So, we do not have to worry about overcounting them.
 
 Choose a pair of pants $\p$. Let $l(\p)$ be the sum of the lengths of its boundary components. Then we claim that
 \[
  l_{max}(\p) \leq l(\p) + s_X
 \]
for some constant $s_X$ depending only on the metric $X$ on $\S$.
Cut $\p$ into two right-angled hexagons, and consider just one of them. Suppose $s$ is a seam edge and that it is adjacent to boundary edges $a$ and $b$, and opposite boundary edge $c$. Let $Z_a(r)$ and $Z_b(r)$ be collar neighborhoods of radius $r$ about $a$ and $b$, respectively. Let $r_a$ and $ r_b$ be the largest radii so that these collar neighborhoods are embedded. Then 
\[
 r_a \leq \sinh^{-1}(\frac{Area(\p)}{l(a)}) \mbox{ and } r_b \leq \sinh^{-1}(\frac{Area(\p)}{l(b)})
\]
Note that $Z_a(r_a) \cup Z_b(r_b)$ covers at least two of the seam edges. In fact, there is a path between the endpoints of $s$ of length at most $l(a) + l(b) + l(c) + r_a + r_b$ (Figure \ref{fig:BoundaryNeighborhoods}).


\begin{figure}[h!]
 \centering
 \includegraphics[scale=1.2]{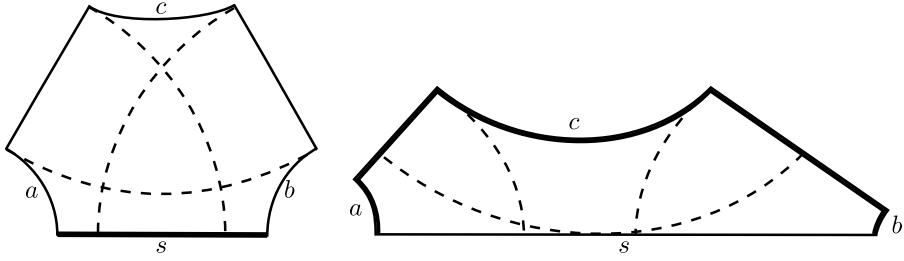}
 \caption{The possible curves joining the endpoints of $s$.}
 \label{fig:BoundaryNeighborhoods}
\end{figure}

 Let $l_{sys}$ be the length of the systole in $X$ and let $s_X = 3 \sinh^{-1}(\frac{Area(X)}{l_{sys}})$. Since $\sinh^{-1}(x)$ is an increasing function, we have
\[
 l(s) \leq s_X + l(\p)
\] 
Thus, $l_{max}(\p) \leq l(\p) + s_X$ for all pairs of pants embedded in $\S$.

Fix $L$ and $K$. Let $l_0 = \frac{\sqrt 3 L}{4 \sqrt K}$.  Note that if $l_0 \geq l_{max}(\p)$ then
  \[
  \frac{1}{8l_{max}(\p)}L > \sqrt {\frac{K}{12}}
  \]
This implies that $\min \{\frac{1}{8 l_{max}(\p)} L, \sqrt{\frac{K}{12}}\} = \sqrt{ \frac{K}{12}}$. Since $K \geq 12$, then $l_0 \geq l_{max}(\p)$ also implies that $L \geq 8 l_{max}(\p)$. Thus, we can simplify inequality (\ref{eq:SumOfPants}) as follows:
 \[
  \#\G^c(L,K) \geq \sum_
  {\substack{\p \subset \S \\ l_{max}(\p) \leq l_0}}
2 + \frac{1}{2} 2^{\sqrt{\frac{ K}{12}}}
 \]

 By \cite{Mirzakhani08}, the number of pairs of pants $\p$ with length at most $L$ grows asymptotically like $c(X)L^{6g-6+2n}$. So there is some constant $c'(X)$ so that this number is bounded below by $c'(X)L^{6g-6+2n}$ for all $L > l_{sys}$, where $l_{sys}$ is the length of the shortest closed geodesic on $\S$. 
 If $l(\p) < l_0 - s_X$ then $l_{max}(\p) < l_0$. The number of pairs of pants with $l(\p) < l_0 -s_X$ is at least $c'(X)(l_0 - s_X)^{6g-6+2n}$ (whenever $l_0 - s_X > l_{sys}$). So there are at least $c'(X)(l_0 - s_X)^{6g-6+2n}$ pairs of pants so that $l_{max}(\p) < l_0$.  So whenever $l_0 - s_X > l_{sys}$ and $K > 12$, we have
 \begin{equation}
 \label{eq:Messy}
  \# \G^c(L,K) \geq  c'(X)(l_0 - s_X)^{6g-6+2n}(2 + \frac 12 2^{\sqrt{\frac{ K}{12}}})
 \end{equation}
for some constant $c'(X)$ depending only on the metric $X$.

Note that $l_0 - s_X> l_{sys}$ if and only if $L > \frac{4}{\sqrt 3} (s_X + l_{sys}) \sqrt K$. Let
\[
 s'_X = s_X + l_{sys}
\]
Then we get inequality (\ref{eq:Messy}) whenever $L \geq \frac{4}{\sqrt 3} s'_X K$ and $K \geq 12$. 
We had $l_0 = \frac{\sqrt 3 L}{4 \sqrt K}$. Note that $\frac{4}{\sqrt 3} < 3$, so we replace $\frac{4}{\sqrt 3}$ by 3 to use nicer numbers in the statement of the theorem.  Furthermore, if $L > 6 s'_X \sqrt K$, then $\frac{L}{3 \sqrt K} - s_X \geq \frac{L}{6\sqrt K}$. Thus, inequality (\ref{eq:Messy}) implies
\[
 \# \G^c(L,K) \geq  c'(X)(\frac{L}{6\sqrt K})^{6g-6+2n}(2 + \frac 12 2^{\sqrt{\frac{ K}{12}}})
\]

\end{proof}

  \bibliographystyle{alpha}
 \bibliography{recount,researchForRecount}
 \end{document}